\documentclass[10pt]{amsart}
\usepackage{amsfonts}
\usepackage{ifthen}
\usepackage{amsthm}
\usepackage{amsmath,mathrsfs}
\usepackage{graphicx}
\usepackage{amscd,amssymb,amsthm}
\usepackage{color}
\usepackage{hyperref}
\usepackage{array,multirow,makecell}

\setcellgapes{1pt}
\makegapedcells
\newcolumntype{R}[1]{>{\raggedleft\arraybackslash }b{#1}}
\newcolumntype{L}[1]{>{\raggedright\arraybackslash }b{#1}}
\newcolumntype{C}[1]{>{\centering\arraybackslash }b{#1}}
\newcounter{minutes}
\divide\time by 60
\newcounter{hours}
\multiply\time by 60 \addtocounter{minutes}{-\time}

\newtheorem{theorem}{Theorem}
\newtheorem{lemma}{Lemma}
\newtheorem{define}{Definition}
\newtheorem{corollary}{Corollary}
\newtheorem{remark}{Remark}
\newtheorem{example}{Example}

\title[Geometric properties of the Fox-Wright functions]{ geometric properties of a certain class of functions related to the Fox-Wright functions}

\author[K. Mehrez]{Khaled Mehrez}
\address{Khaled Mehrez. D\'epartement de Math\'ematiques, Facult\'e de Sciences de Tunis, Universit\'e Tunis El Manar, Tunisia.}
\address{D\'epartement de Math\'ematiques ISSAT Kasserine, Universit\'e de Kairouan, Tunisia}
\email{k.mehrez@yahoo.fr}



\begin{document}

\def\thefootnote{}
\footnotetext{ \texttt{File:~\jobname .tex,
          printed: \number\year-0\number\month-\number\day,
          \thehours.\ifnum\theminutes<10{0}\fi\theminutes}
} \makeatletter\def\thefootnote{\@arabic\c@footnote}\makeatother

\maketitle

\begin{abstract}
The purpose of this paper is to provide a set of sufficient conditions so that the normalized form of the Fox-Wright functions  have certain geometric properties like close-to-convexity, univalency, convexity and starlikeness inside the unit disc. In particular, we  study some geometric properties for some class of functions related to the  generalized hypergeometric functions. 
\end{abstract}

\section{ Introduction}
Let $\mathcal{H}$  denote the class of analytic functions inside the unit disc $\mathcal{D}=\left\{z\in\mathbb{C},\;|z|<1\right\},$ and $\mathcal{A}$ denote the class of analytic functions inside the unit disk $\mathcal{D},$ having the form 
 \begin{equation}
f(z)=z+\sum_{k=2}^\infty a_k z^k,
\end{equation}
where $a_k\in\mathbb{C}$ for all $k\geq2.$ A function $f$ is said to be univalent in a domain $\mathcal{D}$ if it is one-to-one in $\mathcal{D}.$ Further let $\mathcal{S}$ the class of all functions in $\mathcal{A}$  which are univalent in the unit disc $\mathcal{D}.$ A function $f\in\mathcal{A}$ is said called starlike (with respect to the origin $0$), if $tw\in f(\mathcal{D})$ whenever $w\in f(\mathcal{D})$ and $t\in[0,1].$ The class of starlike function is denoted by $\mathcal{S}^*.$ The analytic characterization of the class of starlike function, is given by \cite{D}:
\begin{equation*}
f\in\mathcal{S}^*,\;\;\;\textrm{if\;and\;only\;if,}\;\;\;\;\Re\left(\frac{zf^\prime(z)}{f(z)}\right)>0,\;\;\textrm{for\;all\:}\;z\in\mathcal{D}.
\end{equation*}
Moreover, a function $f\in\mathcal{A}$ is called starlike function of order $\alpha,$ denoted by $\mathcal{S}^*(\alpha),$ if 
$$
\Re\left(\frac{zf^\prime(z)}{f(z)}\right)>\alpha,\;\;\textrm{for\;all\:}\;z\in\mathcal{D}.$$

A function  $f\in\mathcal{A}$ is called convex, denoted by $\mathcal{C},$ if $f$ is univalent in $\mathcal{D}$ and $f(\mathcal{D})$ is a convex domain. The analytic characterization of the class of convex function is
given by:
\begin{equation*}
f\in\mathcal{C},\;\;\;\textrm{if\;and\;only\;if,}\;\;\;\;\Re\left(1+\frac{zf^{\prime\prime}(z)}{f^\prime(z)}\right)>0,\;\;\textrm{for\;all\:}\;z\in\mathcal{D}.
\end{equation*}
If in addition,
$$
\Re\left(1+\frac{zf^{\prime\prime}(z)}{f^\prime(z)}\right)>\alpha,\;\;\textrm{for\;all\:}\;z\in\mathcal{D},$$
where $\alpha\in[0,1)$,  then $f$ is called convex of order $\alpha.$ We denote the class of convex functions of order $\alpha$ by $\mathcal{C}(\alpha).$

An analytic function $f\in\mathcal{A}$ is said to be close-to-convex with respect to a convex function $\varphi:\mathcal{D}\rightarrow\mathbb{C},$ if 
$$
\Re\left(\frac{f^{\prime}(z)}{\varphi^\prime(z)}\right)>0,\;\;\textrm{for\;all\:}\;z\in\mathcal{D}.$$

Given a number $\alpha\in[0,1),$ we say that  $f:\mathbb{D}\rightarrow\mathbb{C}$ is close-to-convex of order $\alpha$ with respect to a convex function $\varphi:\mathbb{D}\rightarrow\mathbb{C},$ if
$$
\Re\left(\frac{f^{\prime}(z)}{\varphi^\prime(z)}\right)>\alpha,\;\;\textrm{for\;all\:}\;z\in\mathcal{D}.$$
It is easy to verify that for all $\alpha\in[0,1)$ 
$$\mathcal{S}^*(\alpha)\subseteq \mathcal{S}(0)=\mathcal{S},\;\; \mathcal{C}(\alpha)\subseteq\mathcal{C}(0)=\mathcal{C}.$$

Recently, several researchers studied families of analytic functions involving special functions,  especially for generalized, Gaussian and Kummer hypergeometric functions \cite{MT, PR},  Wright function \cite{PA}, Mittag-Lefﬂer function \cite{PA1}, and determined  sufficient conditions on the parameters  for these functions to belong to a certain class of univalent functions, such as convex, starlike, close-to-convex. The goal of the present paper is to study some geometric properties for a class of functions related to the Fox-Wright function.

The Fox-Wright function play an important role in various branches of
applied mathematics and engineering sciences. The surprising use of this class of functions has prompted renewed interest in function
theory in the last few decades.  Their properties have been investigated by many authors (see for examples \cite{MMMM, Khaled1, Khaled2, TPP}).

Here, and in what follows, we use ${}_p\Psi_q[.]$ to denote The Fox-Wright (generalized hypergeometric ) function with $p$ numerator parameters $a_1,...,a_p$ and $q$ denominator parameters $b_1,...,b_q,$  which are defined by \cite[p. 4, Eq. (2.4)]{Z}
\begin{equation}\label{11}
{}_p\Psi_q\Big[_{(b_1,B_1),...,(b_q,  B_q)}^{(a_1,A_1),...,(a_p,  A_p)}\Big|z \Big]{}_p\Psi_q\Big[_{(b_q,  B_q)}^{(a_p,  A_p)}\Big|z \Big]=\sum_{k=0}^\infty\frac{\prod_{i=1}^p\Gamma(a_i+kA_i)}{\prod_{j=1}^q\Gamma(b_j+kB_j)}\frac{z^k}{k!},
\end{equation}
$$\left(a_i, b_j\in\mathbb{C},\;\textrm{and}\;\;A_i,B_j\in\mathbb{R}^+\; (i=1,...,p,j=1,...,q)\right).$$
 The convergence conditions and convergence radius of the series at the right-hand side of (\ref{11}) immediately follow from the known asymptotic of the Euler Gamma--function. The defining series in (\ref{11}) converges in the whole complex $z$-plane when
\begin{equation}
\Delta=\sum_{j=1}^q B_j-\sum_{i=1}^p A_i>-1.
\end{equation}
If $\Delta=-1,$ then the series in (\ref{11}) converges for $|z|<\rho,$ and $|z|=\rho$ under the condition $\Re(\mu)>\frac{1}{2},$ where 
\begin{equation}
\rho=\left(\prod_{i=1}^p A_i^{-A_i}\right)\left(\prod_{j=1}^q B_j^{B_j}\right),\;\;\mu=\sum_{j=1}^q b_j-\sum_{k=1}^p a_k+\frac{p-q}{2}
\end{equation}
If, in the definition (\ref{11}), we set
$$A_1=...=A_p=1\;\;\;\textrm{and}\;\;\;B_1=...=B_q=1,$$
we get the relatively more familiar generalized hypergeometric function ${}_pF_q[.]$ given by
\begin{equation}\label{rrrrr}
{}_p F_q\left[^{a_1,...,a_p}_{b_1,...,b_q}\Big|z\right]:=\frac{\prod_{j=1}^q\Gamma(b_j)}{\prod_{i=1}^p\Gamma(a_i)}{}_p\Psi_q\Big[_{(b_1,1),...,(b_q,1)}^{(a_1,1),...,(a_p,1)}\Big|z \Big]
\end{equation}
The main purpose  of this paper is to investigate certain criteria for the univalence, starlikeness, convexity and close-to-convexity for the normalized form of the Fox-Wright function:
\begin{equation}\label{12}
\begin{split}
{}_p\tilde{\Psi}_q\Big[_{(b_q,  B_q)}^{( a_p, A_p)}\Big|z \Big]&=\frac{\prod_{j=1}^q\Gamma(b_j)}{\prod_{i=1}^p\Gamma(a_i)}\sum_{k=0}^\infty\frac{\prod_{i=1}^p\Gamma(a_i+kA_i)}{\prod_{j=1}^q\Gamma(b_j+kB_j)}\frac{z^{k+1}}{k!}\\
&=\sum_{k=0}^\infty U_k(a_p,b_q,A_p,B_q) z^{
k+1},
\end{split}
\end{equation}
where,
$$U_k(a_p,b_q,A_p,B_q)=\frac{\prod_{j=1}^q\Gamma(b_j)}{\prod_{i=1}^p\Gamma(a_i)}\frac{\prod_{i=1}^p\Gamma(a_i+kA_i)}{k!\prod_{j=1}^q\Gamma(b_j+kB_j)},\;k\geq0.$$
 
Each of the following definition will be used in our investigation.

\begin{define}  An inﬁnite sequence $\left\{b_n\right\}_{n\geq1}$ of complex numbers will be called a subordinating factor sequence if whenever
$$f(z)=z+\sum_{n=2}^\infty a_n z^n,$$
is analytic, univalent and convex in $\mathcal{D},$ then
$$\left\{\sum_{n=1}^\infty a_n b_nz^n,\;z\in\mathcal{D}\right\}\subseteq f(\mathcal{D}).$$
\end{define}
\section{Some Lemmas}

In order to prove our results the following preliminary results will be helpful. The
first result is due to S. Ozaki \cite{O}.

\begin{lemma}\label{l1}\cite{O} Let $f(z)=z+\sum_{k=2}^\infty A_k z^k.$ If  $1\leq 2A_2\leq...\leq nA_n\leq(n+1)A_{n+1}\leq ...\leq2,$ or 
$1\geq 2A_2\geq...\geq nA_n\geq(n+1)A_{n+1}\geq ...\geq0,$ then f is close-to-convex with respect to $-\log(1-z).$
\end{lemma}

In the next Lemma, we state the following known condition of univalence.

\begin{lemma}\label{l2}\cite{D,K,O} If $f:\mathcal{D}\rightarrow\mathbb{C},$ is a close-to-convex function, , then it is univalent in $\mathcal{D}.$
\end{lemma}

\begin{lemma}\label{l3}\cite{LF} If the function $g(z)=\sum_{k\geq1} \alpha_k z^k,$ where $\alpha_1=1$ and $\alpha_k\geq0$ for all $k\geq2,$ is analytic in $\mathcal{D},$ and the sequences $\left\{n\alpha_n-(n+1)\alpha_{n+1}\right\}_{n\geq1}, \left\{n\alpha_n\right\}_{n\geq1}$ both are decreasing, then $g$ is starlike in $\mathcal{D}.$
\end{lemma}

\begin{lemma}\label{l4}\cite{LF} If the function $h(z)=\sum_{k\geq1} \beta_k z^{k-1},$ where $\beta_1=1$ and $\beta_k\geq0$ for all $k\geq2,$ is analytic in $\mathcal{D}$ and if the sequence $(\beta_k)_{k\geq1}$ is a convex decreasing sequence, i.e., $\beta_k-2\beta_{k+1}+\beta_{k+2}\geq0$ and  $\beta_k-\beta_{k+1}\geq0,$ for each $k\geq1,$ then 
$$\Re(h(z))>1/2,_;\;\textrm{for\;\;all}\;\;z\in\mathcal{D}.$$
\end{lemma}
\begin{lemma}\label{l5} Let $a,b,A>0$ such that $b\geq a.$ Then the function $H$ defined by 
$$H(z)=\frac{\Gamma(a+Az)}{\Gamma(b+Az)}-\frac{\Gamma(a+A+Az)}{\Gamma(b+A+Az)},$$
is non-negative and decreasing on $(0,\infty).$
\end{lemma}
\begin{proof} Differentiation yields
\begin{equation}\label{souhed1}
\begin{split}
H^\prime(z)&=\frac{A\Gamma(a+Az)}{\Gamma(b+Az)}\left[\psi(a+Az)-\psi(b+Az)\right]\\
&-\frac{A\Gamma(a+A+Az)}{\Gamma(b+A+Az)}\left[\psi(a+A+Az)-\psi(b+A+Az)\right],
\end{split}
\end{equation}
where $\psi$ is the digamma function, defined by $\psi(z)=\frac{\Gamma^\prime(z)}{\Gamma(z)}.$ On the other hand, due to log-convexity property of the Gamma function $\Gamma(z),$ the ratio $x\mapsto \frac{\Gamma(x+\alpha)}{\Gamma(x)}$ is increasing on $(0,\infty),$ when $\alpha>0.$ This implies that the following inequality
\begin{equation}\label{e1}
\frac{\Gamma(x+\alpha)}{\Gamma(x+\alpha+\beta)}\leq \frac{\Gamma(x)}{\Gamma(x+\beta)},
\end{equation}
hold true for all $\alpha,\beta,z>0.$ Setting $x=a+Az,\;\alpha=A,$ and $\beta=b-a$ in (\ref{e1}), we get 
\begin{equation}\label{souhed2}
\frac{\Gamma(a+A+Az)}{\Gamma(b+A+Az)}\leq\frac{\Gamma(a+Az)}{\Gamma(b+Az)}.
\end{equation}
Hence, by using the fact that the digamma function $\psi(z)$ is increasing on $(0,\infty)$ and in view of inequalities (\ref{souhed1}) and (\ref{souhed2}), we obtain
\begin{equation}\label{souhed3}
\begin{split}
H^\prime(z)&\leq\frac{A\Gamma(a+A+Az)}{\Gamma(b+A+Az)}\left[\psi(b+A+Az)-\psi(a+A+Az)+\psi(a+Az)-\psi(b+Az)\right]\\
&=\frac{\Gamma(a+A+Az)}{\Gamma(b+A+Az)} \Phi_{a,b}^A(z),\;\textrm{(say.)}
\end{split}
\end{equation}
By using the Legendre's formula
$$\psi(z)=-\gamma+\int_0^1\frac{t^{z-1}-1}{t-1}dt,$$
where $\gamma$ is the Euler-Mascheroni constant, we have
\begin{equation}\label{souhed4}
\Phi_{a,b}^A(z)=\int_0^1\frac{t^{Az-1}(1-t^A)(t^a-t^b)}{t-1}dt\leq0.
\end{equation}
Finally, in view of (\ref{souhed3}) and (\ref{souhed4}), we deduce that the function $H(z)$ is decreasing on $(0,\infty).$ This ends the proof.
\end{proof}

\begin{lemma}\label{l6}\cite{WI}
The sequences $\left\{\alpha_k\right\}_{k\geq1}$ is a subordinating factor sequence if and only if
$$\Re\left(1+2\sum_{k=1}^\infty\alpha_k z^k\right)>0,\;\;z\in\mathcal{D}.$$
\end{lemma}
\begin{lemma}\label{l7}\cite{MO} If $f\in\mathcal{A}$ and satisfy $|f^\prime(z)-1|<1$ for each $z\in\mathcal{D},$ then $f$ is convex on 
$$\mathcal{D}_{\frac{1}{2}}=\left\{z\in\mathbb{C},\;|z|<\frac{1}{2}\right\}.$$
\end{lemma}

The next Lemma is given in \cite{TH}.

\begin{lemma}\label{l8} If $f\in\mathcal{A}$ and satisfy $|(f(z)/z)-1|<1$ for each $z\in\mathcal{D},$ then $f$ is starlike in 
$\mathcal{D}_{\frac{1}{2}}.$
\end{lemma}

A proof for the following Lemma can be found in \cite[Corollary 1.2]{MO}.

\begin{lemma}\label{l9}If $f\in\mathcal{A}$ and satisfy $|(f(z)/z)-1|<2/\sqrt{5}$ for each $z\in\mathcal{D},$ then $f$ is starlike in 
$\mathcal{D}.$
\end{lemma}
\section{Main results}

In the first main results,  we investigate certain criteria for the univalence and  close-to-convexity of the Fox-Wright functions ${}_p\tilde{\Psi}_q\left[z\right].$

\begin{theorem}\label{T1}The following assertions are true.\\
\noindent 1. If $0<a_i\leq b_i,\;i=1,...,p,$ then the Fox-Wright function ${}_p\tilde{\Psi}_p\left[^{(a_p,A_p)}_{(b_p,A_p)}\Big| z\right]$  is close--to-convex with respect to starlike function $-\log(1-z)$ in $\mathcal{D},$ and consequently it is univalent in $\mathcal{D}.$ \\
\noindent 2. Assume that $0<a_i\leq b_i, A_i\leq B_i$ such that $b_i>x^*,\;i=1,...,p$ where $x^*\approx 1.461632144...,$ is the abscissa of the minimum of the Gamma function. Then the Fox-Wright function ${}_p\tilde{\Psi}_p\left[^{(a_p,A_p)}_{(b_p,B_p)}\Big| z\right]$  is close--to-convex with respect to starlike function $-\log(1-z)$ in $\mathcal{D},$ and consequently it is univalent in $\mathcal{D}.$\\
3. Suppose that $b_j>x^*>a\; (\textrm{or}\;b_j>a>x^*)$ and $B_j\geq1$ for all $1\leq j\leq q,$ such that $\prod_{j=1}^q\Gamma(b_j+B_j)\geq a\prod_{j=1}^q\Gamma(b_j)$ Then the function ${}_1\tilde{\Psi}_q\left[^{(a,1)}_{(b_q,B_q)}\Big| z\right]$  is close--to-convex with respect to starlike function $-\log(1-z)$ in $\mathcal{D},$ and consequently it is univalent in $\mathcal{D}.$
\end{theorem}
\begin{proof}
Upon setting 
$$U_k^{(0)}(a_p,b_q,A_p,B_q)=k U_k(a_p,b_q,A_p,B_q),$$
and
\begin{equation*}
\begin{split}
U_k^{(0)}(a_p, b_q, A, A)&=k U_k(a_p, b_q, A_p, A_q)\\
&=\frac{\prod_{j=1}^q\Gamma(b_j)}{(k-1)!\prod_{i=1}^p\Gamma(a_i)}\frac{\prod_{i=1}^p\Gamma(a_i+kA_i)}{\prod_{j=1}^q\Gamma(b_j+kA_j)},\;k\geq1.
\end{split}
\end{equation*}
Therefore,
\begin{equation}\label{wlw0}
\begin{split}
\frac{U_{k+1}^{(0)}(a_p, b_p, A_p, A_p)}{U_k^{(0)}(a_p, b_p, A_p, A_p)}&=\prod_{i=1}^p\frac{\Gamma(a_i+kA_i+A_i)\Gamma(b_i+kA_i)}{k\Gamma(a_i+kA_i)\Gamma(b_i+kA_i+A_i)}\\
&\leq\prod_{i=1}^p\frac{\Gamma(a_i+kA_i+A_i)\Gamma(b_i+kA_i)}{\Gamma(a_i+kA_i)\Gamma(b_i+kA_i+A_i)}.
\end{split}
\end{equation}
Choosing $x=a_i+kA, \alpha=A_i$ and $\beta=b_i-a_i$ in (\ref{e1}), we obtain 
\begin{equation}\label{wlw}
\frac{\Gamma(a_i+kA_i+A_i)}{\Gamma(b_i+kA_i+A_i)}\leq\frac{\Gamma(a_i+kA_i)}{\Gamma(b_i+kA_i)}
\end{equation}
Combining (\ref{wlw0}) and (\ref{wlw}) we obtain  $$U_{k+1}^{(0)}(a_p,b_p,A_p,A_p)\leq U_{k}^{(0)}(a_p, b_p, A_p, A_p),\;k\geq1.$$ 
On the other hand, setting $x=a_i, \alpha=A$ and $\beta=b_i-a_i$ in (\ref{e1}), we gave 
$$U_{1}^{(0)}(a_p, b_p, A_p, A_p)\leq U_{0}^{(0)}(a_p, b_p, A_p, A_p)=1.$$
This show that that the sequences $(U_k^{(0)}(a_p, b_p, A_p, A_p))_{k\geq0}$ is decreasing, and consequently the function ${}_p\tilde{\Psi}_p\left[^{(\alpha_p,A_p)}_{(\beta_p,A_p)}\Big| z\right]$  is close--to-convex with respect to starlike function $-\log(1-z)$ in $\mathbb{D},$ by Lemma \ref{l1} and  univalent in $\mathbb{D}$ by Lemma \ref{l2}.\\
2. Now, we show that the sequence $(U_k^{(0)}(a_p,  b_p, A_p, B_p))_{k\geq0}$ is decreasing. Again, using the inequality (\ref{e1}) when $x=a_i+kA_i, a=A_i$ and $b=b_i-a_i+k(B_i-A_i),$ we find that 
\begin{equation*}
\Gamma(b_i+kB_i)\Gamma(a_i+kA_i+A_i)\leq \Gamma(a_i+kA_i)\Gamma(b_i+kB_i+A_i).
\end{equation*}
Moreover, from the above inequality and using fact that the Gamma function $\Gamma(z)$ is increasing in $(x^*,\infty),$ we thus obtain 
\begin{equation}\label{e2}
\Gamma(b_i+kB_i)\Gamma(a_i+kA_i+A_i)\leq \Gamma(a_i+kA_i)\Gamma(b_i+kB_i+B_i),
\end{equation}
for all $0<a_i\leq b_i,\;A_i\leq B_i$ such that $b_i>x^*.$ On the other hand, we have 
\begin{equation}\label{e3}
\begin{split}
\frac{U_{k+1}^{(0)}(a_p, b_p, A_p, B_p)}{U_k^{(0)}(a_p, b_p, A_p, B_p)}&=\prod_{i=1}^p\frac{\Gamma(a_i+kA_i+A_i)\Gamma(b_i+kB_i)}{k\Gamma(a_i+kA_i)\Gamma(b_i+kB_i+B_i)}\\
&\leq\prod_{i=1}^p\frac{\Gamma(a_i+kA_i+A_i)\Gamma(b_i+kA_i)}{\Gamma(a_i+kA_i)\Gamma(b_i+kB_i+B_i)}.
\end{split}
\end{equation}
Keeping (\ref{e2}) and (\ref{e3}) in mind, we deduce that the sequences $(U_k^{(0)}(a_p, b_p, A_p, B_p))_{k\geq1}$ is decreasing. Moreover, letting $x=a_i, \alpha=A_i$ and $\beta=b_i-a_i$ in (\ref{e1}), we get
$$\frac{\Gamma(b_i)\Gamma(a_i+A_i)}{\Gamma(a_i)\Gamma(b_i+A_i)}\leq1,$$
and using the fact that $\Gamma(b_i+B_i)\geq\Gamma(b_i+A_i),\;b_i>x^*$ we deduce that 
$$\frac{\Gamma(b_i)\Gamma(a_i+A_i)}{\Gamma(a_i)\Gamma(b_i+B_i)}\leq1,\;\textrm{for}\;1\leq i\leq p,$$
and consequently the sequences $(U_k^{(0)}(a_p, b_p, A_p, B_p))_{k\geq0}$ is decreasing. Therefore, by Lemma \ref{l1} we deduce that the Fox-Wright function ${}_p\Psi_p\left[^{(\alpha_i,A_i)}_{(\beta_i,B_i)}\Big| z\right]$  is close--to-convex with respect to starlike function $-\log(1-z)$ in $\mathbb{D}$ and consequently is univalent in $\mathbb{D}$ by means of Lemma \ref{l2}.\\
\noindent 3.  We define the sequence $(U_k^1)_{k\geq1}$ by 
\begin{displaymath}
 U_k^{(1)}= \left\{ \begin{array}{ll}
 \frac{\Gamma(a+k)}{(k-1)!\Gamma(a)}\prod_{j=1}^q\frac{\Gamma(b_j)}{\Gamma(b_j+kB_j)},& k\geq1\\
1,& k=0
\end{array} \right.
\end{displaymath}
The condition $\prod_{j=1}^q\Gamma(b_j+B_j)\geq a\prod_{j=1}^q\Gamma(b_j)$ show that $U_0^{(1)}\geq U_1^{(1)}.$ Next, by using the fact that 
\begin{equation}
\begin{split}
\Gamma(b_j+kB_j)&=\Gamma(b_j+(k-1)B_j+B_j)\\
&\geq\Gamma(b_j+(k-1)B_j+1)\\
&=(b_j+(k-1)B_j)\Gamma(b_j+(k-1)B_j),\;k\geq1,
\end{split}
\end{equation}
we get 
\begin{equation*}
\begin{split}
U_k^{(1)}-U_{k+1}^{(1)}&=\frac{\Gamma(a+k)\prod_{j=1}^q\Gamma(b_j)}{(k-1)!\Gamma(a)}\left[\frac{1}{\prod_{j=1}^q\Gamma(b_j+kB_j)}-\left(\frac{a+k}{k}\right)\frac{1}{\prod_{j=1}^q\Gamma(b_j+kB_j+B_j)}\right]\\
&\geq\frac{\Gamma(a+k)\prod_{j=1}^q\Gamma(b_j)}{(k-1)!\Gamma(a)\prod_{j=1}^q\Gamma(b_j+kB_j)}\left[1-\left(\frac{a+k}{k}\right)\frac{1}{\prod_{j=1}^q(b_j+kB_j)}\right]\\
&=\frac{\Gamma(a+k)\prod_{j=1}^q\Gamma(b_j)}{(k-1)!\Gamma(a)\prod_{j=1}^q\Gamma(b_j+kB_j+1)}\left[\prod_{j=1}^q(b_j+kB_j)-\left(\frac{a+k}{k}\right)\right]\\
&\geq\frac{\Gamma(a+k)\prod_{j=1}^q\Gamma(b_j)}{(k-1)!\Gamma(a)\prod_{j=1}^q\Gamma(b_j+kB_j+1)}\left[(a+k)^q-\left(\frac{a+k}{k}\right)\right]\\
&\geq0,
\end{split}
\end{equation*}
for all $k\geq1.$ This implies that  the sequence $(U_k^{(1)})_{k\geq1}$ is decreasing, and consequently $(U_k^{(1)})_{k\geq0}$ is decreasing.  By Lemma \ref{l1} and Lemma \ref{l2} we deduce that the function ${}_1\tilde{\Psi}_q\left[^{(a,1)}_{(b_q,B_q)}\Big| z\right]$  is close--to-convex with respect to starlike function $-\log(1-z)$ in $\mathbb{D},$ and it is univalent in $\mathbb{D},$  which evidently completes the proof of Theorem \ref{T1}.
\end{proof}
\begin{theorem}\label{T2} Suppose that $b_j\geq 2a\geq2$ and $B_j\geq2$ for all $1\leq j\leq q,$ such that $\prod_{j=1}^q\Gamma(b_j+B_j)\geq 2a\prod_{j=1}^q\Gamma(b_j).$ Then the function ${}_1\tilde{\Psi}_q\left[^{(a,1)}_{(b_q,B_q)}\Big| z\right]$  is starlike in $\mathcal{D}.$
\end{theorem}
\begin{proof} We apply Lemma \ref{l3} to prove that the function ${}_1\tilde{\Psi}_q\left[^{(a,1)}_{(b_q,B_q)}\Big| z\right]$  is starlike in $\mathcal{D}.$  In the proof of Part 3 of Theorem \ref{T1}, we get that the sequence $(U_k^1)_{k\geq1}$ is decreasing under the conditions $b_j\geq 2a>2>x^*$ and $B_j\geq2>1$ for all $1\leq j\leq q,$ such that $\prod_{j=1}^q\Gamma(b_j+B_j)\geq a\prod_{j=1}^q\Gamma(b_j).$ Moreover, we gave
\begin{equation*}
\begin{split}
U_k^1-2U_{k+1}^1&=\frac{\Gamma(a+k)\prod_{j=1}^q\Gamma(b_j)}{(k-1)!\Gamma(a)}\left[\frac{1}{\prod_{j=1}^q\Gamma(b_j+kB_j)}-2\left(\frac{a+k}{k}\right)\frac{1}{\prod_{j=1}^q\Gamma(b_j+kB_j+B_j)}\right]\\
&\geq\frac{\Gamma(a+k)\prod_{j=1}^q\Gamma(b_j)}{(k-1)!\Gamma(a)\prod_{j=1}^q\Gamma(b_j+kB_j)}\left[1-2\left(\frac{a+k}{k}\right)\frac{1}{\prod_{j=1}^q(b_j+kB_j)}\right]\\
&=\frac{\Gamma(a+k)\prod_{j=1}^q\Gamma(b_j)}{(k-1)!\Gamma(a)\prod_{j=1}^q\Gamma(b_j+kB_j+1)}\left[\prod_{j=1}^q(b_j+kB_j)-2\left(\frac{a+k}{k}\right)\right]\\
&\geq\frac{\Gamma(a+k)\prod_{j=1}^q\Gamma(b_j)}{(k-1)!\Gamma(a)\prod_{j=1}^q\Gamma(b_j+kB_j+1)}\left[(2(a+k))^q-2\left(\frac{a+k}{k}\right)\right]\\
&\geq0,
\end{split}
\end{equation*}
This implies that $U_k^1-2U_{k+1}^1+U_{k+2}^1\geq0$ for each $k\geq1$ and the condition $\prod_{j=1}^q\Gamma(b_j+B_j)\geq 2a\prod_{j=1}^q\Gamma(b_j)$ implies that $U_1^1-2U_{2}^1+U_{3}^1\geq0.$ So, Lemma \ref{l3} completes the proof of Theorem \ref{T2}.
\end{proof}
\begin{theorem}\label{T3}Suppose that $b_i>a_i>0$ and $A_i>0$ for all $i\in\left\{1,...,p\right\}.$ Then
\begin{equation}
\Re\left(z^{-1}{}_{p+1}\tilde{\Psi}_p\Big[_{( b_p,  A_p)}^{(1,1),(a_p, A_p)}\Big|z \Big]\right)>\frac{1}{2},\;z\in\mathcal{D}.
\end{equation}
\end{theorem}
\begin{proof}For convenience, let us write
$$z^{-1}
{}_{p+1}\tilde{\Psi}_p\Big[_{( b_p, A_p)}^{(1,1),( a_p,  A_p)}\Big|z \Big]=\sum_{k=1}^\infty U_k^{(2)}z^{k-1},$$
\textrm{where}\;
$$U_k^{(2)}=\prod_{i=1}^p\frac{\Gamma(b_i)\Gamma(a_i+(k-1)A_i)}{\Gamma(a_i)\Gamma(b_i+(k-1)A_i)},\;k\geq1.$$
Firstly, we set $z=a_i+(k-1)A_i,\;\alpha=a_i$ and $\beta=b_i-a_i$ in (\ref{e1}), we obtain
\begin{equation}\label{7ab1}
\begin{split}
 U_k^{(2)}- U_{k+1}^{(2)}=\prod_{i=1}^p\frac{\Gamma(b_i)}{\Gamma(a_i)}\left[\frac{\Gamma(a_i+(k-1)A_i)}{\Gamma(b_i+(k-1)A_i)}-\frac{\Gamma(a_i+kA_i)}{\Gamma(b_i+kA_i)}\right]\geq0,
\end{split}
\end{equation}
for $k\geq1.$ Secondly, we have 
\begin{equation}\label{7ab2}
\begin{split}
 U_k^{(2)}- 2U_{k+1}^{(2)}+U_{k+2}^{(2)}&=U_k^{(2)}- U_{k+1}^{(2)}+U_{k+2}^{(2)}-U_{k+1}^{(2)}\\
&=\prod_{i=1}^p\frac{\Gamma(b_i)(H(a_i,b_i,k-1)-H(a_i,b_i,k))}{\Gamma(a_i)}\\
&\geq0,
\end{split}
\end{equation}
by means of Lemma \ref{l5}. Keeping in mind  (\ref{7ab1}) and (\ref{7ab2}) and applying Lemma \ref{l4} we deduce that the statement asserted in Theorem \ref{T3} holds. 
\end{proof}

The following result  follows in view of Theorem \ref{T3} and Lemma \ref{l6}.

\begin{corollary}Keeping the notation and constraints of hypotheses of Theorem \ref{T3}. Then, the sequence 
$$\left\{\prod_{i=1}^p\frac{\Gamma(b_i)\Gamma(a_i+(k-1)A_i)}{\Gamma(a_i)\Gamma(b_i+(k-1)A_i)}\right\}_{k\geq1},$$
is a subordinating factor sequence for the class $\mathcal{C}.$
\end{corollary}

\begin{remark} We define the function ${}_p\Psi_q^*[z]$ by
$${}_p\Psi_q^*\left[^{(a_p,A_p)}_{(b_q,B_q)}\Big| z\right]=\frac{\prod_{j=1}^q\Gamma(b_j)}{\prod_{i=1}^p\Gamma(a_i)}{}_p\Psi_q\left[^{(a_p,A_p)}_{(b_q,B_q)}\Big| z\right].$$
It is clear that 
$${}_p\Psi_q^*\left[^{(a_p,A_p)}_{(b_q,B_q)}\Big| z\right]=z^{-1}{}_p\tilde{\Psi}_q\left[^{(a_p,A_p)}_{(b_q,B_q)}\Big| z\right].$$
By using  the differentiation formula
\begin{equation}
\left({}_p\Psi_q^*\left[^{(a_p,A_p)}_{(b_q,B_q)}\Big| z\right]\right)^\prime=\prod_{i=1}^q\frac{\Gamma(a_i+A_i)}{\Gamma(a_i)}\prod_{j=1}^q\frac{\Gamma(b_j)}{\Gamma(b_j+B_j)}{}_p\Psi_q^*\left[^{(a_p+A_p,A_p)}_{(b_q+B_q,B_q)}\Big| z\right]
\end{equation}
and Theorem \ref{T3}, we deduce that 
\begin{equation}
\Re\left(\left({}_{p+1}\Psi_q^*\left[^{(0,1), (a_p-A_p,A_p)}_{(b_p-A_p,A_p)}\Big| z\right]\right)^\prime\right)>\frac{1}{2}\prod_{i=1}^p\frac{\Gamma(b_i)\Gamma(a_i+A_i)}{\Gamma(a_i)\Gamma(b_i+A_i)},\;\;z\in\mathcal{D},
\end{equation}
for all $b_i>a_i>0$ and $A_i>0.$ In addition, we note that the ratios
$$\frac{1}{2}\prod_{i=1}^p\frac{\Gamma(b_i)\Gamma(a_i+A_i)}{\Gamma(a_i)\Gamma(b_i+A_i)},$$
is in $[0,1).$
\end{remark}

\begin{theorem}\label{T8} Let $a_i, b_
i, A_i, B_i>0$ such that $\sum_{i=1}^pA_i=\sum_{j=1}^p B_j$ and $\min(a_i/A_i)\geq1.$ Assume that the H-function $H_{p,p}^{p,0}\left[_{(A_i,a_i)}^{(B_i,b_i)}\right]$ is non-negative. If the following inequality
\begin{equation}
\prod_{j=1}^{p}\left[\frac{\Gamma(a_i+A_i)}{\Gamma(b_i+B_i)}-\frac{\Gamma(a_i+2A_i)}{\rho\Gamma(b_i+2B_i)}(1-e^\rho)\right]\leq\prod_{j=1}^{p}\frac{\Gamma(a_i)}{\Gamma(b_i)},
\end{equation}
holds true for all $p\geq1,$ then the function 
${}_p\tilde{\Psi}_{p+1}\left[^{(a_p,A_p)}_{(2,1),(b_p,B_p)}\Big|z\right]$
is convex  in $\mathcal{D}_{1/2}.$
\end{theorem}
\begin{proof} A simple computation gives
\begin{equation}\label{EEE}
\begin{split}
\left|\frac{\partial}{\partial z}{}_p\tilde{\Psi}_{p+1}\left[^{(a_p,A_p)}_{(2,1),(b_p,B_p)}\Big|z\right]-1\right|&=\left|\sum_{k=1}^\infty\prod_{i=1}^p\frac{\Gamma(b_i)\Gamma(a_i+kA_i)}{\Gamma(a_i)\Gamma(b_i+kB_i)}\frac{z^k}{(k!)^2}\right|\\
&=\left|\sum_{k=0}^\infty\prod_{i=1}^p\frac{\Gamma(b_i)\Gamma(a_i+A_i+kA_i)}{\Gamma(a_i)\Gamma(b_i+B_i+kB_i)}\frac{z^{k+1}}{(k+1)!)^2}\right|\\
&\leq\sum_{k=0}^\infty\prod_{i=1}^p\frac{\Gamma(b_i)\Gamma(a_i+A_i+kA_i)}{\Gamma(a_i)\Gamma(b_i+B_i+kB_i)}\frac{|z|^{k+1}}{k!}\\
&=|z|\prod_{i=1}^p \frac{\Gamma(b_i)}{\Gamma(a_i)}{}_p\Psi_p\left[^{(a_p+A_p,A_p)}_{(B_p+B_p,B_p)}\Big||z|\right]\\
&\leq\prod_{i=1}^p\frac{\Gamma(b_i)}{\Gamma(a_i)} {}_p\Psi_p\left[^{(a_p+A_p,A_p)}_{(b_p+B_q,B_p)}\Big|1\right],\;z\in\mathcal{D}.
\end{split}
\end{equation}
In view of Luke's type inequality of the Fox-Wright function \cite[Remark 9, Eq. (469)]{Khaled1}
\begin{equation}\label{luke}
{}_p\Psi_p\left[^{(a_p,A_p)}_{(b_p,B_p)}\Big|z\right]\leq \psi_{0,0}-\frac{\psi_{0,1}}{\rho}(1-e^{\rho z}),\;z\in\mathbb{R},
\end{equation}
where $$\psi_{0,0}=\prod_{i=1}^p\frac{\Gamma(a_i)}{\Gamma(b_i)},\;\psi_{0,1}=\prod_{i=1}^p\frac{\Gamma(a_i+A_i)}{\Gamma(b_i+B_i)},$$
and the inequality (\ref{EEE}), we gave
\begin{equation}\label{EEE1}
\begin{split}
\left|\frac{\partial}{\partial z}{}_p\tilde{\Psi}_p\left[^{(a_p,A_p)}_{(b_p,B_p)}\Big|z\right]-1\right|&\leq \prod_{i=1}^p\frac{\Gamma(b_i)}{\Gamma(a_i)}\left[\frac{\Gamma(a_i+A_i)}{\Gamma(b_i+B_i)}-\frac{\Gamma(a_i+2A_i)}{\rho\Gamma(b_i+2B_i)}(1-e^\rho)\right]\\
&\leq 1.
\end{split}
\end{equation}
 With this, we deduce that the function ${}_p\tilde{\Psi}_p\left[^{(a_p,A_p)}_{(b_p,B_p)}\Big|z\right]$
is convex  in $\mathcal{D}_{1/2}$ by means of Lemma \ref{l7}. The proof of Theorem \ref{T8} is complete.
\end{proof}

The following example follows from Theorem \ref{T8} combined with \cite[Corollary 4]{Khaled3}.

\begin{example} Let $\alpha, \beta$ and $\gamma$ be a real numbers and satisfies the following conditions
$$\alpha\in(0,1], \frac{1}{\alpha}-1<\beta, \gamma+\beta\geq\frac{1}{2},\frac{1}{\alpha}=1+\frac{1}{\alpha\beta}.$$
Setting $a_1=b_2=A_1=2B_2=1, a_2=\frac{\gamma+\beta}{\alpha\beta}, b_1=1+\frac{\gamma}{\beta},A_2=\frac{1}{2\alpha\beta}.$ If the following inequality
$$\prod_{j=1}^{2}\left[\frac{\Gamma(a_i+A_i)}{\Gamma(b_i+B_i)}-\frac{\Gamma(a_i+2A_i)}{\rho_1\Gamma(b_i+2B_i)}(1-e^{\rho_1})\right]\leq\prod_{j=1}^{2}\frac{\Gamma(a_i)}{\Gamma(b_i)},$$ 
holds true\;\textrm{where}\;$$\rho_1=\left(\frac{1}{2}\right)^{\frac{1}{2}}\left(\frac{1}{2\alpha}\right)^{\frac{1}{2\alpha}}\left(\frac{1}{2\alpha\beta}\right)^{\frac{1}{2\alpha\beta}},$$
then the function 
$${}_2\tilde{\Psi}_3\left[^{(1,1),(\frac{\gamma+\beta}{\alpha\beta},\frac{1}{2\alpha\beta})}_{(2,1),(1+\frac{\gamma}{\beta},\frac{1}{2\alpha}),(1,\frac{1}{2})}\Big|z\right]$$
is convex in $\mathcal{D}_{\frac{1}{2}}.$ 
\end{example}

In \cite[Remark 2]{Khaled1}, the author was proved that the H-function $H_{p,p}^{p,0}\left[_{(A,a_i)}^{(A,b_i)}\right]$ is non-negative under the hypotheses 
$$(H_2):\;0<a_1\leq...\leq a_p,\;0<b_1\leq...\leq b_p,\;\;\sum_{j=1}^k b_j-\sum_{j=1}^k a_j\geq0,\;\textrm{for}\;k=1,...,p.$$
Obviously, by repeating  the procedure of the proofs of the above Theorem, when it is used Theorem 9 in \cite{Khaled1}, we can deduce the following result:

\begin{theorem}\label{TT9}Under the hypotheses $(H_2)$ such that the following inequality
\begin{equation}\label{YYY}
\prod_{j=1}^{p}\left[\frac{\Gamma(a_i+A)}{\Gamma(b_i+A)}-\frac{\Gamma(a_i+2A)}{\Gamma(b_i+2A)}(1-e)\right]\leq\prod_{j=1}^{p}\frac{\Gamma(a_i)}{\Gamma(b_i)},
\end{equation}
holds true for all $p\geq1,$ then the function 
${}_p\tilde{\Psi}_{p+1}\left[^{(a_p,A)}_{(2,1),(b_p,A)}\Big|z\right]$
is convex  in $\mathcal{D}_{1/2}.$
\end{theorem}
\begin{corollary}\label{C2}Suppose that $b_i,a_i>0$ such that $$b_i>\frac{a_i-1+\sqrt{(a_i+1)^2+4a_i(e-1)(a_i+1)}}{2}\;\textrm{for}\;1\geq i\geq p.$$ Then the normalized hypergeometric function ${}_p\tilde{F}_p$ defined by
$${}_p\tilde{F}_{p+1}\left[^{a_1,...,a_p}_{2,b_1,...,b_p}\Big|z\right]=z{}_pF_{p+1}\left[^{a_1,...,a_p}_{2,b_1,...,b_p}\Big|z\right],$$
is convex in $\mathcal{D}_{\frac{1}{2}}.$
\end{corollary}
\begin{proof} Choosing $A=1$ in Theorem \ref{TT9}, we obtain that (\ref{YYY}) is equivalent to
$$b_i^2+(1-a_i)b_i+a_i(-1+(1-e)(a_i+1))>0,$$
where $b_i\geq a_i>0.$ This in turn implies that the following inequality $$b_i>\frac{a_i-1+\sqrt{(a_i+1)^2+4a_i(e-1)(a_i+1)}}{2},$$ holds true. With this, the proof of Corollary \ref{C2} is complete.
\end{proof}
\begin{theorem}Under the assumption and statements of Theorem \ref{T8}. The function ${}_p\tilde{\Psi}_{p}\left[^{(a_p,A_p)}_{(b_p,B_p)}\Big|z\right]$
is starlike  in $\mathcal{D}_{1/2}.$
\end{theorem}
\begin{proof}A computation gives
\begin{equation}\label{EEE1}
\begin{split}
\left|\left({}_p\tilde{\Psi}_p\left[^{(a_p,A_p)}_{(b_p,B_p)}\Big|z\right]/z\right)-1\right|&=\left|\sum_{k=1}^\infty \prod_{i=1}^p\frac{\Gamma(b_i)\Gamma(a_i+kA_i)}{\Gamma(a_i)\Gamma(b_i+kB_i)}\frac{z^k}{k!}\right|\\
&=\left|\sum_{k=0}^\infty \prod_{i=1}^p\frac{\Gamma(b_i)\Gamma(a_i+A_i+kA_i)}{\Gamma(a_i)\Gamma(b_i+B_i+kB_i)}\frac{z^{k+1}}{(k+1)!}\right|\\
&\leq \left|z\sum_{k=0}^\infty \prod_{i=1}^p\frac{\Gamma(b_i)\Gamma(a_i+A_i+kA_i)}{\Gamma(a_i)\Gamma(b_i+B_i+kB_i)}\frac{z^{k}}{k!}\right|\\
&=\left|z \prod_{i=1}^p \frac{\Gamma(b_i)}{\Gamma(a_i)}{}_p\Psi_p\left[^{(a_p+A_p,A_p)}_{(B_p+B_p,B_p)}\Big|z\right]\right|,
\end{split}
\end{equation}
Now, by applying Luke's formula (\ref{luke}) once more with (\ref{EEE1}), we obtain that
\begin{equation}\label{EEE11}
\begin{split}
\left|\left({}_p\tilde{\Psi}_p\left[^{(a_p,A_p)}_{(b_p,B_p)}\Big|z\right]/z\right)-1\right|&\leq \prod_{i=1}^p \frac{\Gamma(b_i)}{\Gamma(a_i)}{}_p\Psi_p\left[^{(a_p+A_p,A_p)}_{(B_p+B_p,B_p)}\Big|1\right]\\
&\leq1,
\end{split}
\end{equation}
for all $|z|<1.$ Then, the function ${}_p\tilde{\Psi}_p\left[^{(a_p,A_p)}_{(b_p,B_p)}\Big|z\right]$
is starlike  in $\mathcal{D}_{1/2}$ by using Lemma \ref{l8}.
\end{proof}

The following example follows from Theorem \ref{T8} combined with \cite[Corollary 4]{Khaled3}.

\begin{example} Let $\alpha, \beta$ and $\gamma$ be a real numbers and satisfies the following conditions
$$\alpha\in(0,1], \frac{1}{\alpha}-1<\beta, \gamma+\beta\geq\frac{1}{2},\frac{1}{\alpha}=1+\frac{1}{\alpha\beta}.$$
Setting $a_1=b_2=A_1=2B_2=1, a_2=\frac{\gamma+\beta}{\alpha\beta}, b_1=1+\frac{\gamma}{\beta},A_2=\frac{1}{2\alpha\beta}.$ If the following inequality
$$\prod_{j=1}^{2}\left[\frac{\Gamma(a_i+A_i)}{\Gamma(b_i+B_i)}-\frac{\Gamma(a_i+2A_i)}{\rho_1\Gamma(b_i+2B_i)}(1-e^{\rho_1})\right]\leq\prod_{j=1}^{2}\frac{\Gamma(a_i)}{\Gamma(b_i)},$$ 
holds true\;\textrm{where}\;$$\rho_1=\left(\frac{1}{2}\right)^{\frac{1}{2}}\left(\frac{1}{2\alpha}\right)^{\frac{1}{2\alpha}}\left(\frac{1}{2\alpha\beta}\right)^{\frac{1}{2\alpha\beta}},$$
then the function 
$${}_2\tilde{\Psi}_2\left[^{(1,1),(\frac{\gamma+\beta}{\alpha\beta},\frac{1}{2\alpha\beta})}_{\;\;\;\;(1+\frac{\gamma}{\beta},\frac{1}{2\alpha}),(1,\frac{1}{2})}\Big|z\right]$$
is  starlike in $\mathcal{D}_{\frac{1}{2}}.$ 
\end{example}

The proof of the following claim follows by repeating the same calculations in Theorem \ref{T8}.

\begin{theorem}\label{TTt9} Keeping the notation and constraints of Theorem \ref{TT9}. Then the function 
${}_p\tilde{\Psi}_{p}\left[^{(a_p,A)}_{(b_p,A)}\Big|z\right]$
is starlike  in $\mathcal{D}_{1/2}.$
\end{theorem}
\begin{corollary}Let $(a_i)_{1\leq i\leq p}$ be a real numbers such that 
$$0<a_i<\frac{2-e+\sqrt{e^2+4e-4}}{2(e-1)}.$$ Then, the function $K_1$ defined by
$$K_1\left[(a_i)_{1\leq i\leq p}|z\right]=\sum_{k=0}^\infty\prod_{i=1}^p\frac{a_i}{a_i+k}\frac{z^{k+1}}{k!},$$
is starlike in $\mathcal{D}_{\frac{1}{2}}.$ In particular, the function $K_2$ defined by
 $$K_2\left(a;z\right)=\sum_{k=0}^\infty\frac{a^p}{(a+k)^p}\frac
{z^{k+1}}{k!},$$
is starlike in $\mathcal{D}_{\frac{1}{2}}$ for each $0<a<\frac{2-e+\sqrt{e^2+4e-4}}{2(e-1)}.$
\end{corollary}
\begin{proof}Setting $bi=a_i+1$ and $A=1$ in Theorem \ref{TT9}, we get that (\ref{YYY}) is equivalent to
$$(1-e)a_i^2+(2-e)a_i+2>0,$$
which clearly holds since $0<a_i<\frac{2-e+\sqrt{e^2+4e-4}}{2(e-1)}$ and hence the function $K_1$ is starlike in $\mathcal{D}_{\frac{1}{2}}$ by means of Theorem \ref{TTt9}. Finally, choosing $a_i=a$ we get that the function $K_2$ is convex in $\mathcal{D}_{\frac{1}{2}}.$
\end{proof}
\begin{corollary}Assume that the hypotheses of Corollary \ref{C2} are satisfied.
Then the normalized hypergeometric function ${}_p\tilde{F}_p$ defined by
$${}_p\tilde{F}_p\left[^{a_1,...,a_p}_{b_1,...,b_p}\Big|z\right]=z{}_pF_{p}\left[^{a_1,...,a_p}_{b_1,...,b_p}\Big|z\right],$$
is starlike in $\mathcal{D}_{1/2}.$
\end{corollary}
\begin{proof}The claim follows from Theorem \ref{TTt9} by repeating the same calculations in the proof of Corollary \ref{C2}.
\end{proof}
\begin{theorem}\label{TY8} Let $a_i, b_
i, A_i, B_i>0$ such that $\sum_{i=1}^pA_i=\sum_{j=1}^p B_j$ and $\min(a_i/A_i)\geq1.$ Assume that the H-function $H_{p,p}^{p,0}\left[_{(A_i,a_i)}^{(B_i,b_i)}\right]$ is non-negative. If the following inequality
\begin{equation}\label{RRR}
\prod_{j=1}^{p}\left[\frac{\Gamma(a_i+A_i)}{\Gamma(b_i+B_i)}-\frac{\Gamma(a_i+2A_i)}{\rho\Gamma(b_i+2B_i)}(1-e^\rho)\right]\leq\frac{2}{\sqrt{5}}\prod_{j=1}^{p}\frac{\Gamma(a_i)}{\Gamma(b_i)},
\end{equation}
holds true, then the function 
$$z\mapsto{}_p\tilde{\Psi}_{p}\left[^{(a_p,A_p)}_{(b_p,B_p)}\Big|z\right]$$
is starlike in $\mathcal{D}.$
\end{theorem}
\begin{proof}By means of Lemma \ref{l9}, and keeping (\ref{EEE1}) and (\ref{RRR}) in mind we obtain the desired result.  It is important to mention here that there is another proof. Namely, upon setting 
$$p(z)=\left(z\frac{\partial}{\partial z} \left({}_p\tilde{\Psi}_p\left[^{(a_p,A_p)}_{(b_p,B_p)}\Big|z\right)\right)\right]\Big/\left({}_p\tilde{\Psi}_p\left[^{(a_p,A_p)}_{(b_p,B_p)}\Big|z\right]\right),\;z\in\mathcal{D}.$$
Then $p$ is analytic and $p(0)=1.$ To prove the result, we need to show  that $\Re(p(z))>0$ for all $z\in\mathcal{D}.$ It is easy to see that, if $|p(z)-1|<1,$ then $\Re(p(z))>0$ for all $z\in\mathcal{D}.$ A simple computation shows that
\begin{equation}\label{!!!}
\begin{split}
\frac{\partial}{\partial z} \left({}_p\tilde{\Psi}_p\left[^{(a_p,A_p)}_{(b_p,B_p)}\Big|z\right]\right)-\frac{1}{z}\;{}_p\tilde{\Psi}_p\left[^{(a_p,A_p)}_{(b_p,B_p)}\Big|z\right]&=\sum_{k=1}^\infty \prod_{i=1}^p\frac{\Gamma(b_i)\Gamma(a_i+kA_i)}{\Gamma(a_i)\Gamma(b_i+kB_i)}\frac{z^k}{(k-1)!}\\
&=\sum_{k=0}^\infty \prod_{i=1}^p\frac{\Gamma(b_i)\Gamma(a_i+A_i+kA_i)}{\Gamma(a_i)\Gamma(b_i+B_i+kB_i)}\frac{z^{k+1}}{k!}\\
&=z\;\prod_{i=1}^p\frac{\Gamma(b_i)}{\Gamma(a_i)}{}_p\Psi_p\left[^{(a_p+A_p,A_p)}_{(b_p+B_p,B_p)}\Big|z\right].
\end{split}
\end{equation}
By using the fact that
$$\left|\frac{1}{z}\;{}_p\tilde{\Psi}_p\left[^{(a_p,A_p)}_{(b_p,B_p)}\Big|z\right]\right|\geq1,\;\textrm{for}\;|z|<1,$$
and from (\ref{!!!}), we obtain 
\begin{equation}
\begin{split}
\left|p(z)-1\right|&\leq \prod_{i=1}^p\frac{\Gamma(b_i)}{\Gamma(a_i)}{}_p\Psi_p\left[^{(a_p+A_p,A_p)}_{(b_p+B_p,B_p)}\Big|1\right]\\
&\leq1,
\end{split}
\end{equation}
which evidently completes the proof of Theorem \ref{TY8}.
\end{proof}
\begin{example} Let $\alpha, \beta$ and $\gamma$ be a real numbers and satisfies the following conditions
$$\alpha\in(0,1], \frac{1}{\alpha}-1<\beta, \gamma+\beta\geq\frac{1}{2},\frac{1}{\alpha}=1+\frac{1}{\alpha\beta}.$$
Setting $a_1=b_2=A_1=2B_2=1, a_2=\frac{\gamma+\beta}{\alpha\beta}, b_1=1+\frac{\gamma}{\beta},A_2=\frac{1}{2\alpha\beta}.$ If the following inequality
$$\prod_{j=1}^{2}\left[\frac{\Gamma(a_i+A_i)}{\Gamma(b_i+B_i)}-\frac{\Gamma(a_i+2A_i)}{\rho_1\Gamma(b_i+2B_i)}(1-e^{\rho_1})\right]\leq\frac{2}{\sqrt{5}}\prod_{j=1}^{2}\frac{\Gamma(a_i)}{\Gamma(b_i)},$$ 
holds true, then the function
$$z\mapsto{}_2\tilde{\Psi}_2\left[^{(1,1),(\frac{\gamma+\beta}{\alpha\beta},\frac{1}{2\alpha\beta})}_{\;\;\;\;(1+\frac{\gamma}{\beta},\frac{1}{2\alpha}),(1,\frac{1}{2})}\Big|z\right]$$
is starlike in $\mathcal{D}.$
\end{example}

The following Theorem can be derived by repeating the proof of the above Theorem when we used the inequality (\ref{RRR1}).

\begin{theorem} Assume that the hypotheses $(H_2)$ are satisfied. In addition, suppose that the following inequality
\begin{equation}\label{RRR1}
\prod_{j=1}^{p}\left[\frac{\Gamma(a_i+A)}{\Gamma(b_i+A)}-\frac{\Gamma(a_i+2A)}{\Gamma(b_i+2A)}(1-e)\right]\leq\frac{2}{\sqrt{5}}\prod_{j=1}^{p}\frac{\Gamma(a_i)}{\Gamma(b_i)},
\end{equation}
is valid. Then the function $$z\mapsto{}_p\tilde{\Psi}_{p}\left[^{(a_p,A_p)}_{(b_p,B_p)}\Big|z\right]$$
is starlike in $\mathcal{D}.$
\end{theorem}
\begin{corollary} Let $b_i\geq a_i>0$ such that 
$$b_i>\frac{-2+\sqrt{5}a_i+\sqrt{(\sqrt{5}a_i-2)^2+8a_i\sqrt{5}(e(a_i+1)-a_i)}}{4},\;\;\textrm{for all}\;\; 1\leq i\leq p .$$ Then, the normalized hypergeometric function ${}_p\tilde{F}_p$ is starlike in $\mathcal{D}.$
\end{corollary}
\begin{proof}Letting $A=1$ in the above Theorem, we obtain that the inequality (\ref{RRR1}) is equivalent to 
$$2b_i^2+(2-\sqrt{5}a_i)b_i+\sqrt{5}a_i(-1+(1-e)(a_i+1))>0,\;\;\textrm{for all}\;\; 1\leq i\leq p .$$
This in turn implies that
$$b_i>\frac{-2+\sqrt{5}a_i+\sqrt{(\sqrt{5}a_i-2)^2+8a_i\sqrt{5}(e(a_i+1)-a_i)}}{4},\;\;\textrm{for all}\;\; 1\leq i\leq p .$$
\end{proof}


\end{document}